\documentclass[a4paper]{amsart}
\usepackage[latin1]{inputenc}
\usepackage{amssymb}
\usepackage{amsmath}
\usepackage{mathrsfs}
\usepackage{eufrak}
\usepackage{amsthm}
\usepackage{amsfonts}
\usepackage{textcomp}
\usepackage{graphicx}
\usepackage[pdftex]{color}
\usepackage{paralist}
\usepackage[shortlabels]{enumitem}
\usepackage{hyperref}
\usepackage{verbatim}
\usepackage{comment}
\usepackage[arrow, matrix, curve]{xy}
\usepackage{tikz} \usetikzlibrary{matrix,patterns,shapes,decorations.pathmorphing,decorations.pathreplacing,calc,arrows} \usepackage{tikz-cd} 
\usepackage{listings}
\usepackage{color}

\definecolor{dkgreen}{rgb}{0,0.6,0}
\definecolor{gray}{rgb}{0.5,0.5,0.5}
\definecolor{mauve}{rgb}{0.58,0,0.82}

\newtheorem*{corollary*}{Corollary}
\newtheorem*{conjecture*}{Conjecture}
\newtheorem*{example*}{Example}
\newtheorem*{theorem*}{Theorem}
\newtheorem*{proposition*}{Proposition}

\newtheorem{theorem}{Theorem}[section]

\newtheorem{corollary}[theorem]{Corollary}
\newtheorem{lemma}[theorem]{Lemma}

\newtheorem*{claim*}{Claim}

\newtheorem*{question}{Question}

\theoremstyle{definition}

\theoremstyle{remark}

\numberwithin{equation}{section}

\makeatletter
\renewcommand*\env@matrix[1][\
arraystretch]{%
  \edef\arraystretch{#1}%
  \hskip -\arraycolsep
  \let\@ifnextchar\new@ifnextchar
  \array{*\c@MaxMatrixCols c}}
\makeatother

\renewcommand{\mod}{\operatorname{mod}}

\newcommand{\Ext}{\operatorname{Ext}}

\newcommand{\domdim}{\operatorname{domdim}}

\newcommand{\gldim}{\operatorname{gldim}}
\newcommand{\End}{\operatorname{End}}

\newcommand{\Hom}{\operatorname{Hom}}
\newcommand{\add}{\operatorname{\mathrm{add}}}

\begin{document}

\title{A Cluster tilting module for a representation-infinite block of a group algebra}
\date{\today}

\subjclass[2010]{Primary 16G10, 16E10}

\keywords{Cluster tilting modules, algebras of quaternion type, group algebras}

\author{Bernhard B\"ohmler}
\address{FB Mathematik, TU Kaiserslautern, Gottlieb-Daimler-Str. 48, 67653 Kaiserslautern, Germany}
\email{boehmler@mathematik.uni-kl.de}

\author{Ren\'{e} Marczinzik}
\address{Institute of algebra and number theory, University of Stuttgart, Pfaffenwaldring 57, 70569 Stuttgart, Germany}
\email{marczire@mathematik.uni-stuttgart.de}

\begin{abstract}
Let $G=SL(2,5)$ be the special linear group of $2 \times 2$-matrices with coefficients in the field with $5$ elements.
We show that the principal block over a splitting field $K$ of characteristic two of the group algebra $KG$ has a $3$-cluster tilting module.
This gives the first example of a representation-infinite block of a group algebra having a cluster tilting module and answers a question by Erdmann and Holm.
\end{abstract}

\maketitle

\section*{Introduction}
We assume that algebras are finite dimensional over a field $K$.
For $n \geq 1$, an $A$-module $M$ is called an \emph{$n$-cluster tilting module} if it satisfies:
\begin{align*}
    \add(M) &= \{X \in \mod-A \mid \Ext_A^i(M,X)=0 \ \text{for} \ 1 \leq i \leq n-1 \} \\
    &= \{X \in \mod-A \mid \Ext_A^i(X,M)=0  \ \text{for} \ 1 \leq i \leq n-1 \}.
\end{align*}
We remark that in some references like \cite{EH} an $n$-cluster tilting module is called an maximal $(n-1)$-orthogonal module.
The concept of $n$-cluster tilting modules was introduced by Iyama in \cite{Iya} and \cite{Iya4}. It has found several important applications, for instance in the theory of  cluster algebras, see \cite{GLS}.
Cluster tilting modules are especially important for selfinjective algebras where recent methods allow to construct many examples related to other structures in algebra and combinatorics, see for example \cite{DI} and \cite{CDIM}.
One of the most important classes of selfinjective algebras are group algebras. This leads to the following natural question:
\begin{question}
When does a block of a group algebra have a cluster tilting module?

\end{question}
In \cite{EH}, Erdmann and Holm showed that selfinjective algebras with a cluster tilting module have complexity at most one (recently it was shown in \cite{MV} that this result does not hold for non-selfinjective algebras). Erdmann and Holm used this result in \cite[Section 5.3]{EH} to show that a block of a group algebra can only have a cluster tilting module when it is representation-finite or Morita equivalent to an algebra of quaternion type. 

Every representation-finite block of a group algebra is derived equivalent to a symmetric Nakayama algebra and recently a complete classification for the existence of cluster tilting modules was obtained in \cite[Section 5]{DI} for selfinjective Nakayama algebras. 
For algebras of quaternion type (which are always of infinite representation type) it is however unknown whether they can have cluster tilting modules and this was posed in \cite[Section 5.3]{EH} as an open question.
There is no universal method to construct cluster tilting modules or to show that they do not exist. The fact that the classification of indecomposable modules for algebras of quaternion type is still not known makes the search and verification of cluster tilting modules especially hard. We remark that the existence of an $n$-cluster tilting module is a Morita invariant.
In this article we show that the algebra of quaternion type $Q(3 \mathcal{A})_2^2$ has a $3$-cluster tilting module. This algebra is Morita equivalent to the principal block of the group algebra of $SL(2,5)$ over a splitting field of characteristic two.
Our main result is as follows:
\begin{theorem*}
Let $A$ be the principal block of the group algebra $KG$ for $G=SL(2,5)$ over a splitting field of characteristic two.
Then $A$ has a $3$-cluster tilting module.
\end{theorem*}

This gives the first example of a cluster tilting module of a representation-infinite block of a group algebra and it gives a positive answer to the question of Erdmann and Holm about the existence of cluster tilting modules for algebras of quaternion type.
We remark that we found the $3$-cluster tilting module in the main result by experimenting with the GAP-package \cite{QPA}. The proof also uses the calculation of quiver and relations for an endomorphism ring which was obtained with the aid of the computer.

\section{An example of a 3-cluster tilting module for the algebra of quaternion type $Q(2 \mathcal{A})_2^2$}
We assume that all algebras are finite dimensional over a field $K$ and all modules are finite dimensional right modules unless stated otherwise. $J$ will denote the Jacobson radical of an algebra and $D=\Hom_K(-,K)$ the natural duality. We assume that the reader is familiar with the basics of representation theory and homological algebra of finite dimensional algebras and refer for example to the textbook \cite{SY}. 

The \emph{global dimension} $\gldim A$ of an algebra $A$ is defined as the supremum of all projective dimensions of the simple $A$-modules. It is well known that the global dimension of $A$ coincides with the global dimension of the opposite algebra $A^{op}$, see for example \cite[Exercise 4.1.1]{W}.
The \emph{dominant dimension} $\domdim A$ of $A$ is defined as the minimal $n$ such that $I_n$ is not projective (or infinite if no such $n$ exists), where 
$$0 \rightarrow A \rightarrow I_0 \rightarrow I_1 \rightarrow \cdots $$
is a minimal injective coresolution of the regular $A$-module $A$.
The dominant dimension of $A$ coincides with the dominant dimension of the opposite algebra $A^{op}$, see \cite[Theorem 4]{M}.

We will also need the following lemma on the behaviour of the global and dominant dimension under extensions of the ground field. For a field extension $F$ of $K$, we denote by $A_F:=A \otimes_K F$ the $F$-algebra which is obtained from $A$ by the field extension.
\begin{lemma} \label{field extensionlemma}
Let $A$ be a finite dimensional algebra over the field $K$ and let $F$ be a field extension of $K$. 
\begin{enumerate}
\item $\domdim A=\domdim A_F$.
\item If $A/J$ is separable (where $J$ denotes the Jacobson radical of $A$), then
$\gldim A= \gldim A_F$.

\end{enumerate}

\end{lemma}
\begin{proof}
\leavevmode
\begin{enumerate}
\item See \cite[Lemma 5]{M}.
\item See \cite[Corollary 18]{ERZ}.
\end{enumerate}
\end{proof}
Recall that a module $M$ is a \emph{generator} of $\mod-A$ when every indecomposable projective $A$-module is a direct summand of $M$ and $M$ is a \emph{cogenerator} of $\mod-A$ when every indecomposable injective $A$-module is a direct summand of $M$. 
\begin{theorem} \label{clustertiltingtheorem}
Let $A$ be a non-semisimple connected finite dimensional algebra with an $A$-module $M$ that is a generator and cogenerator of $\mod-A$.
Then $M$ is an $n$-cluster tilting module if and only if $B:=\End_A(M)$ is a higher Auslander algebra of global dimension $n+1$, that is $B$ has global dimension equal to $n+1$ and dominant dimension equal to $n+1$.

\end{theorem}
\begin{proof}
See \cite[Theorem 2.6]{Iya2} for an elementary proof.
\end{proof} 

We refer to \cite[Section VII]{E} for the precise definition of algebras of quaternion type, which arise in the study of blocks of group algebras with quaternion defect groups.
The tables starting at page $303$ of \cite{E} give quiver and relations of algebras of quaternion type.
In this article, we only need to know the algebra of quaternion type $Q(3 \mathcal{A})_2^2$ that we describe next. Let $K$ be a field of characterstic two.
Let $A=KQ/I$ be the following quiver algebra where $Q$ is given by
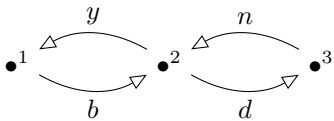
\begin{center}
	\begin{tikzpicture}
		\node (v1) at (-4,0) {$\bullet^{1}$};
		\node (v2) at (-2,0) {$\bullet^{2}$};
		\node (v3) at (0,0) {$\bullet^{3}$};
		
		\draw [-open triangle 45] (v1) edge[bend right] node[below] {$b$}  (v2);
		\draw [-open triangle 45] (v2) edge[bend right] node[above] {$y$}  (v1);
		\draw [-open triangle 45] (v3) edge[bend right] node[above] {$n$}  (v2);
		\draw [-open triangle 45] (v2) edge[bend right] node[below] {$d$}  (v3);
	\end{tikzpicture}
\end{center}
and the relations are given by
$$I=\langle byb-bdnybdn,yby-dnybdny,ndn-nybdnyb,dnd-ybdnybd,bybd,ndny\rangle.$$

This is the algebra of quaternion type $Q(3 \mathcal{A})_2^2$ and this algebra is Morita equivalent to the principal block of the the group algebra $FG$ where $G=SL(2,5)$ is the special linear group of $2 \times 2$-matrices over the field with five elements and $F$ is a splitting field of characteristic two, see for example page $110$ of \cite{H} and section $7$ of \cite{E2}.
The algebra $A$ is symmetric of period $4$ and $\dim_K(A)=36$.
The dimension vectors of the indecomposable projective $A$-modules $P_1, P_2$ and $P_3$ are respectively given by~$[4,4,2],$~$[4,8,4]$  and~$[2,4,4]$. We define the following $A$-modules:
\begin{enumerate}
\item Let $M_1= e_3 A/nA$, which has dimension vector $[0,0,1]$.
\item Let $M_2=e_3 A/nybdnyA$, which has dimension vector $[1,3,3]$.
\item Let $M_3=e_3A/ nyA$, which has dimension vector $[0,1,2]$.
\item Let $M_4=e_2A/yA$, which has dimension vector $[1,4,2]$.
\end{enumerate}
Let $M:=A \oplus M_1 \oplus M_2 \oplus M_3 \oplus M_4$. Note that every indecomposable summand of $M$ has simple top.
We fix $A$ and $M$ as above for the rest of this article.
We show that $M$ is a $3$-cluster tilting module.
\begin{theorem}
Let $A$ be the algebra of quaternion type $Q(3 \mathcal{A})_2^2$ over a field $F$ with characteristic two.
Then $M$ is a 3-cluster tilting module.

\end{theorem}
\begin{proof}
Clearly $M$ is a generator and cogenerator of $\mod-A$. 
We show that $B:=\End_A(M)$ has global dimension $4$ and dominant dimension $4$. Then, $M$ is a $3$-cluster tilting module by Theorem \ref{clustertiltingtheorem}.
First assume that $K$ has two elements.
The following QPA program calculates quiver and relations of $B^{op}$ over the field with two elements and shows that $B^{op}$ has global dimension and dominant dimension equal to $4$. We remark that GAP applies functions from the right. Thus, it calculates the opposite algebra of the endomorphism ring of $M$.
\begin{tiny}
\begin{verbatim}
LoadPackage("qpa");
k:=2;F:=GF(2);Q:=Quiver(3,[[1,2,"b"],[2,3,"d"],[2,1,"y"],[3,2,"n"]]);
kQ:=PathAlgebra(F,Q);AssignGeneratorVariables(kQ);
rel:=[b*y*b-(b*d*n*y)^(k-1)*b*d*n,y*b*y-(d*n*y*b)^(k-1)*d*n*y,
n*d*n-(n*y*b*d)^(k-1)*n*y*b,
d*n*d-(y*b*d*n)^(k-1)*y*b*d,b*y*b*d,n*d*n*y];
A:=kQ/rel; B:=Basis(A);U:=Elements(B);Display(U);n:=Size(B);
UU:=[];for i in [4..n] do Append(UU,[U[i]]);od;                            
t1:=UU[4];
M1:=RightAlgebraModuleToPathAlgebraMatModule(RightAlgebraModule(A, \*, RightIdeal(A,[t1])));
N1:=CoKernel(InjectiveEnvelope(M1));M1:=N1;
t2:=UU[33];M2:=RightAlgebraModuleToPathAlgebraMatModule(RightAlgebraModule(A, \*, RightIdeal(A,[t2])));
N2:=CoKernel(InjectiveEnvelope(M2));M2:=N2;
t3:=UU[10];
M3:=RightAlgebraModuleToPathAlgebraMatModule(RightAlgebraModule(A, \*, RightIdeal(A,[t3])));
N3:=CoKernel(InjectiveEnvelope(M3));M3:=N3;
t4:=UU[3];
M4:=RightAlgebraModuleToPathAlgebraMatModule(RightAlgebraModule(A, \*, RightIdeal(A,[t4])));
N4:=CoKernel(InjectiveEnvelope(M4));M4:=N4;
N:=DirectSumOfQPAModules([N1,N2,N3,N4]);
projA:=IndecProjectiveModules(A);RegA:=DirectSumOfQPAModules(projA);
M:=DirectSumOfQPAModules([RegA,N]);
B:=EndOfModuleAsQuiverAlgebra(M)[3];
QQ:=QuiverOfPathAlgebra(B);Display(QQ);rel:=RelatorsOfFpAlgebra(B);
gd:=GlobalDimensionOfAlgebra(B,33);dd:=DominantDimensionOfAlgebra(B,33);
\end{verbatim}
\end{tiny}

\noindent We observe that $B^{op}=K\hat{Q}/\hat{I}$ is a quiver algebra where $\hat{Q}$ is given by
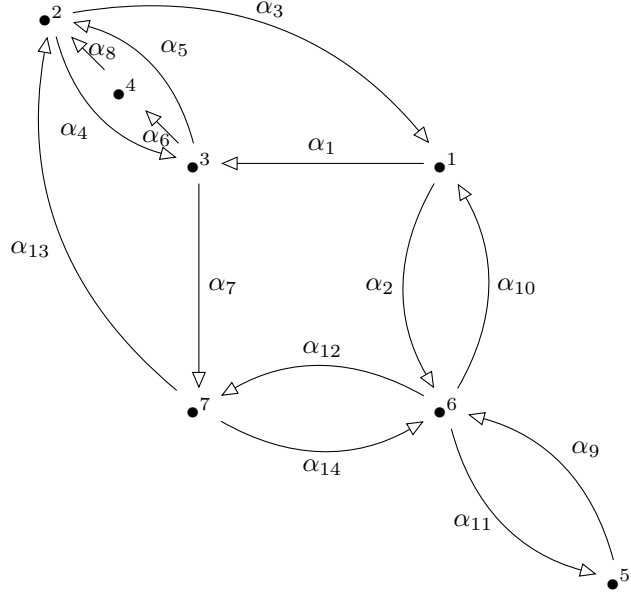
\begin{center}
\begin{tikzpicture}[scale=0.65]

	\node (v1) at (1,2.5) {$\bullet^{1}$};
	\node (v2) at (-7,5.5) {$\bullet^{2}$};
	\node (v3) at (-4,2.5) {$\bullet^{3}$};
	\node (v4) at (-5.5,4) {$\bullet^{4}$};
	\node (v5) at (4.5,-6) {$\bullet^{5}$};
	\node (v6) at (1,-2.5) {$\bullet^{6}$};
	\node (v7) at (-4,-2.5) {$\bullet^{7}$};

	\draw [-open triangle 45] (v1) edge node[above] {$\alpha_1$}  (v3);
	\draw [-open triangle 45] (v1) edge[bend right] node[left] {$\alpha_2$}  (v6);
	\draw [-open triangle 45] (v2) edge[bend left] node[above] {$\alpha_3$}  (v1);
	\draw [-open triangle 45] (v2) edge[bend right] node[below left] {$\alpha_4$}  (v3);
	\draw [-open triangle 45] (v3) edge[bend right] node[above right] {$\alpha_5$}  (v2);
	\draw [-open triangle 45] (v3) edge node[below,pos=0.65] {$\alpha_6$}  (v4);
	\draw [-open triangle 45] (v3) edge node[right] {$\alpha_7$}  (v7);
	\draw [-open triangle 45] (v4) edge node[above,pos=0.05] {$\alpha_8$}  (v2);
	\draw [-open triangle 45] (v5) edge[bend right] node[above right] {$\alpha_9$}  (v6);
	\draw [-open triangle 45] (v6) edge[bend right] node[right] {$\alpha_{10}$}  (v1);
	\draw [-open triangle 45] (v6) edge[bend right] node[left] {$\alpha_{11}$}  (v5);
	\draw [-open triangle 45] (v6) edge[bend right] node[above] {$\alpha_{12}$}  (v7);
	\draw [-open triangle 45] (v7) edge[bend left] node[below left] {$\alpha_{13}$}  (v2);
	\draw [-open triangle 45] (v7) edge[bend right] node[below] {$\alpha_{14}$}  (v6);
	
\end{tikzpicture}
\end{center}
\noindent and the relations are given by
\begin{align*}
\hat{I}=\langle &\alpha_1\alpha_6, \alpha_2\alpha_{11}, \alpha_1\alpha_7 + \alpha_2\alpha_{12}, \alpha_4\alpha_5, \alpha_6\alpha_8 + \alpha_7\alpha_{13}, \alpha_8\alpha_3, \alpha_9\alpha_{10}, \alpha_2\alpha_{10} + \alpha_1\alpha_5\alpha_3, \alpha_2\alpha_{10}\alpha_1, \alpha_3\alpha_1\alpha_5,\\ &\alpha_3\alpha_2\alpha_{10},\alpha_3\alpha_2 + \alpha_4\alpha_7\alpha_{14}, \alpha_5\alpha_3\alpha_1, \alpha_5\alpha_4\alpha_6, \alpha_8\alpha_4\alpha_6, \alpha_9\alpha_{12}\alpha_{13}, \alpha_{12}\alpha_{13} + \alpha_{10}\alpha_1\alpha_5,\\ &\alpha_{10}\alpha_1\alpha_7 + \alpha_{11}\alpha_9\alpha_{12}, \alpha_{10}\alpha_2\alpha_{10} + \alpha_{12}\alpha_{13}\alpha_3, \alpha_{14}\alpha_{12} + \alpha_{13}\alpha_4\alpha_7, \alpha_{14}\alpha_{10}\alpha_2 + \alpha_{14}\alpha_{11}\alpha_9,\\ &\alpha_{13}\alpha_3\alpha_2 + \alpha_{14}\alpha_{12}\alpha_{14}, \alpha_1\alpha_7\alpha_{14}\alpha_{12}, \alpha_2\alpha_{10}\alpha_2\alpha_{10}, \alpha_3\alpha_1\alpha_7\alpha_{14}, \alpha_3\alpha_1 + \alpha_4\alpha_6\alpha_8\alpha_4,\\ &\alpha_7\alpha_{14}\alpha_{12} + \alpha_6\alpha_8\alpha_4\alpha_7, \alpha_5\alpha_4 + \alpha_7\alpha_{14}\alpha_{10}\alpha_1, \alpha_7\alpha_{14}\alpha_{12}\alpha_{13}, \alpha_9\alpha_{12}\alpha_{14}\alpha_{12}, \alpha_{10}\alpha_2\alpha_{10}\alpha_2 + \alpha_{11}\alpha_9\alpha_{11}\alpha_9,\\ &\alpha_{12}\alpha_{13}\alpha_4\alpha_6, \alpha_{12}\alpha_{14}\alpha_{12}\alpha_{13}, \alpha_{14}\alpha_{12}\alpha_{13} + \alpha_{13}\alpha_4\alpha_6\alpha_8, \alpha_{14}\alpha_{10}\alpha_2\alpha_{10}, \alpha_{13}\alpha_3\alpha_1 + \alpha_{14}\alpha_{12}\alpha_{13}\alpha_4,\\ &\alpha_{10}\alpha_2 + \alpha_{11}\alpha_9 + \alpha_{12}\alpha_{14}\alpha_{10}\alpha_1\alpha_7\alpha_{14}, \alpha_{13}\alpha_3 + \alpha_{14}\alpha_{10}\alpha_1\alpha_7\alpha_{14}\alpha_{10}\rangle.
\end{align*}

\noindent With $B^{op}$ also $B$ has global dimension and dominant dimension equal to $4$ and thus $M$ is a $3$-cluster tilting module.
Now let $F$ be an arbitrary field with characteristic two, which is an extension of the field $K$ with two elements.
We have 
$$\End_{A_F}(M \otimes_K F) \cong \End_A(M) \otimes_K F \cong B_F ,$$
which has also dominant and global dimension equal to $4$. This follows from Lemma \ref{field extensionlemma} and the fact that~$B/J$ is separable, since $B$ is a quiver algebra.
Thus $M \otimes_K F$ is also a 3-cluster tilting module of $A_F$.

\end{proof}
We remark that it took the supercomputer "nenepapa" from the TU Kaiserslautern $105$ hours to compute the endomorphism ring of $M$. The data of this supercomputer are as follows. Compute-Server Linux (Gentoo): Dell PowerEdge R730, 2x Intel Xeon E5-2697AV4 2.6 GHz, Turbo 3.60 GHz, 40 MB SmartCache, 32 Cores, 64 Threads, 768 GB RAM.\newline\newline
\indent As remarked earlier, the principal block of the group algebra $KG$ for $G=SL(2,5)$ over a splitting field $K$ of characteristic two is Morita equivalent to the algebra of quaternion type $Q(3 \mathcal{A})_2^2$. As a corollary of the previous Theorem we obtain our main result: 
\begin{corollary}
Let $G=SL(2,5)$ and $K$ be a field of characteristic two that is a splitting field for $KG$.
Then the principal block of $KG$ has a $3$-cluster tilting module.
\end{corollary} 

Note that not every algebra of quaternion type has a cluster tilting module.
In fact, the group algebra $KG$ of the quaternions $G$ of order $8$ over a field $K$ with characteristic two has no cluster tilting modules, since it is representation-infinite and we have $\Ext_{KG}^1(M,M) \neq 0$ for every non-projective $KG$-module $M$ by a result of Tachikawa, see \cite[Theorem 8.6]{T}.
\section*{Acknowledgements} 
We thank Karin Erdmann for having informed us in private communication that she has also found a~$3$-cluster tilting module for another algebra of quaternion type which is not a block of a group algebra. We thank Thorsten Holm for providing a reference to his habilitation thesis.
Bernhard B\"ohmler gratefully acknowledges funding by the DFG (SFB/TRR 195). Ren{\'e} Marczinzik gratefully acknowledges funding by the DFG (with project number 428999796). We profited from the use of the GAP-package \cite{QPA}.


\begin{thebibliography}{Gus}
\bibitem[CDIM20]{CDIM}
Aaron Chan, Erik Darp{\"o}, Osamu Iyama, and Ren{\'e} Marczinzik.
\newblock Periodic trivial extension algebras and fractionally Calabi-Yau algebras.
\newblock \url{https://arxiv.org/abs/2012.11927}.

\bibitem[DI20]{DI}
Erik Darp{\"o} and Osamu Iyama.
\newblock d-representation-finite self-injective algebras.
\newblock {\em Advances in Mathematics}, Volume 362, 2020.
  
\bibitem[ERZ57]{ERZ} 
Samuel Eilenberg, Alex Rosenberg, and Daniel Zelinsky.
\newblock On the dimension of modules and algebras VIII.
\newblock {\em Nagoya Math. J.} Volume 12 , 1957, 71-93.
  
\bibitem[E90]{E} Karin Erdmann. 
\newblock Blocks of tame representation type and related algebras.
\newblock {\em Lecture Notes in Mathematics} 1428, 1990.

\bibitem[E88]{E2} Karin Erdmann.
\newblock Algebras and quaternion defect groups II.
\newblock {\em Mathematische Annalen} Volume 281, pages 561-582, 1988.

\bibitem[EH08]{EH}
Karin Erdmann and Thorsten Holm.
\newblock {Maximal n-orthogonal modules for selfinjective algebras}.
\newblock {\em Proc. Amer. Math. Soc.}, 136(9):3069--3078, 2008.

\bibitem[GLS06]{GLS}
Christof Gei{\ss}, Bernard Leclerc, and Jan Schr{\"o}er.
\newblock Rigid modules over preprojective algebras.
\newblock {\em Inventiones mathematicae}, 165(3):589--632, Sep 2006.

\bibitem[H01]{H} Thorsten Holm.
\newblock Blocks of Tame Representation Type and Related Algebras: Derived Equivalences and Hochschild Cohomology.
\newblock Habilitation thesis 2001, \url{http://www2.iazd.uni-hannover.de/~tholm/habil.ps}.

\bibitem[Iya07a]{Iya}
Osamu Iyama.
\newblock Auslander correspondence.
\newblock {\em Advances in Mathematics}, 210(1):51 -- 82, 2007.

\bibitem[Iya07b]{Iya4}
Osamu Iyama.
\newblock Higher-dimensional Auslander--Reiten theory on maximal orthogonal
  subcategories.
\newblock {\em Advances in Mathematics}, 210(1):22 -- 50, 2007.

\bibitem[Iya08]{Iya2}
Osamu Iyama.
\newblock Auslander-Reiten theory revisited.
\newblock {\em Trends in representation theory of algebras and related topics. Proceedings of the 12th international conference on representations of algebras and workshop } EMS Series of Congress Reports, 349-397 , 2008. 

\bibitem[MV21]{MV} 
Rene Marczinzik, Leartis Vaso.
\newblock Existence of a 2-cluster tilting module does not imply finite complexity.
\newblock \url{https://arxiv.org/abs/2101.05671}.

\bibitem[M68]{M} 
Bruno M\"uller.
\newblock The Classification of Algebras by Dominant Dimension.
\newblock {\em Canadian Journal of Mathematics} , Volume 20 , 1968 , 398 - 409. 


\bibitem[QPA16]{QPA}
The QPA-team.
\newblock {\em {QPA - Quivers, path algebras and representations - a GAP
  package, Version 1.25}}, 2016.

\bibitem[SY11]{SY}
Andrzej Skowro\'{n}ski and Kunio Yamagata.
\newblock {\em Frobenius algebras {I}}.
\newblock EMS Textbooks in Mathematics. European Mathematical Society (EMS),
  Z\"{u}rich, 2011.
\newblock Basic representation theory.


\bibitem[T73]{T} 
Hiroyuki Tachikawa.
\newblock Quasi-Frobenius Rings and Generalizations.
\newblock Lecture Notes in Mathematics, Volume 351, 1973.

\bibitem[W95]{W} Charles Weibel.
\newblock An Introduction to Homological Algebra 
\newblock {\em Cambridge Studies in Advanced Mathematics Book 38}, Cambridge University Press 1995.

\end{thebibliography}
\end{document}